\documentclass[12pt]{article}
\usepackage[utf8]{inputenc}
\usepackage{amsmath}
\usepackage{amssymb}
\usepackage{mathrsfs}
\usepackage{amsthm}
\usepackage{xcolor}
\usepackage{cite}
\usepackage{hyperref}
\usepackage{enumerate}
\usepackage[left=0.7in,right=0.7in,top=0.7in,bottom=0.7in]{geometry}
\usepackage{titlesec}
\titleformat*{\section}{\large\bfseries}
\newtheorem{theorem}{Theorem}[section]
\newtheorem{lemma}[theorem]{Lemma}
\newtheorem{corollary}[theorem]{Corollary}

\newtheorem{example}[theorem]{Example}
\newtheorem{proposition}[theorem]{Proposition}
\newtheorem{remark}[theorem]{Remark}
\numberwithin{equation}{section}
\title{Self-adjoint and co-isometry composition and weighted composition operators on Fock-type spaces}
\author{\large Anuradha Gupta and Geeta Yadav$^*$}
\date{}
\begin{document}
\maketitle
\begin{abstract} 
In this paper we obtain characterizations for adjoint of a composition and weighted composition operator to be composition and weighted composition operator on $F_{\psi}^2,$ respectively. We study the co-isometry composition and weighted composition operators on $F_{\psi}^2.$


\end{abstract}    

\section{Introduction and preliminaries}

Consider set $S=[0, +\infty)$ and let $\psi: S \longrightarrow S$ be a $C^3$- function with
\begin{equation} \label{equation 1}
 \psi'(y)>0, \,\,\,\, \psi''(y)\geq 0 \,\, \text{and} \, \psi'''(y) \geq 0 \,\, \text{for all}\,\, y\in S.
 \end{equation}   
It follows from \eqref{equation 1} that at least $\psi$ should grow as a linear function.
Recall that, a sequence $c^{*}=\{c_{r}^{*}\}_{r \geq 0}$ is Stieltjes moment sequence if
$$c_{r}^{*}=\int_{0}^{\infty} s^{r} \, d\, \nu(s), $$
where $\nu$ is a non-negative measure on $S.$ Stieltjes \cite{Stieltjes} have characterized these sequences in terms of some positive definiteness conditions. Let $d\, \nu(s)=e^{-\psi(s)} ds$ and let sequence $c=\{c_{r}\}_{r \geq 0}$ where
$$c_{r}=\int_{0}^{\infty} s^{r}\, e^{-\psi(s)} \,ds. $$

We denote the complex Euclidean $n$-space by $\mathbb{C}^n$ and by $d \mu$ we denote the Lebesgue measure on $\mathbb{C}^n.$ We define
$$ \langle w,z \rangle= \sum_{j=1}^{n} w_{j} \bar{z_{j}} \,\,\, \text{and}\,\,\,  \langle w,w \rangle=|w|^2 $$
for all $w=(w_{1},w_{2},\cdots,w_{n})$ and $z=(z_{1},z_{2},\cdots,z_{n}) \in \mathbb{C}^n. $  We will use the notation $z=(z_{1},z_{2},\cdots,z_{n})$ or  $z=(z_{1} \, z_{2}\, \cdots z_{n})^{T}$ for $z \in \mathbb{C}^n$ where $T$ denotes the transpose of a matrix.

For $1 \leq p < \infty$ the Fock-type space $F_{\psi}^{p}$ is the Banach space of all holomorphic complex valued functions on $\mathbb{C}^n$ such that
$$||g||_{p}=\Big(\int_{\mathbb{C}^n} |g(z)  e^{-\frac{1}{2} \psi(|z|^2)}|^{p} d\nu(z) \Big)^{\frac{1}{p}} < \infty.$$
In particular, for $p=2$, $F_{\psi}^{2}$ is a Hilbert space with the inner product defined by
$$\langle g,h \rangle_{\psi}=\int_{\mathbb{C}^n} g(z) \overline{h(z)} e^{- \psi(|z|^2)} d\nu(z) $$
for all $g, h \in F_{\psi}^{2}.$

Suppose that the weight function $\psi$ satisfies the additional condition that there exists some real number $l< \frac{1}{2}$ such that
\begin{equation}\label{mild smoothness condition}
 \phi''(y)=O\, (y^{-1/2} [\phi'(y)]^{1+l}), \,\, y \to \infty, 
\end{equation} 
where $\phi(y)=y\, \psi'(y).$ Moreover, when $n>1$ we assume that $\psi$ also satisfies the equation \eqref{mild smoothness condition}.\\

Let $G$ be the function defined by 
$$ G(t)= \sum_{r=0}^{\infty} \frac{t^{r}} {c_{r}} \,\,\, \text{for}\,\, t \in \mathbb{C},$$
this series has an infinite radius of convergence  (see \cite{Hankel operators and the Stieltjes moment}) and the reproducing kernel, $K_{p,\psi}\in F_{\psi}^{2} ,$ for point evaluation at $p \in \mathbb {C}^{n}$ satisfying $f(p)=\langle f, K_{p,\psi} \rangle$ for each $f\in F_{\psi}^{2}$  is
\begin{align*}
K_{p,\psi}(z) &= \frac{1}{(n-1)!} \, G^{(n-1)} (\langle z,p \rangle)  \\
             &=\frac{1}{(n-1)!}\sum_{r=n-1}^{\infty} \frac{r(r-1) \cdots(r-n+2)}{c_{r}} \langle z, p \rangle^{r-n+1}.
\end{align*} 
In the subsequent sections we will repeatedly use the fact that set $Span\{K_{p,\psi} : p \in \mathbb{C}^n \}$ is dense in $F_{\psi}^{2}$ \cite{Bounded and compact}. Note that for $\psi$ being polynomial of degree one, $F_{\psi}^{2}$ coincide with the classical Fock space. Also moments $c_{r}>0$ for all $r \geq 0.$ For more information on moments $c_{r}$ and space $F_{\psi}^{2}$ one can see \cite{Hankel op & Related Bergman kernel}. \\

Clearly, $K_{0,\psi}(w)=\frac{1}{c_{n-1}}$ and  $K_{p,\psi}(0)=\frac{1}{c_{n-1}}.$ Moreover,
\begin{equation} \label{property of conjugate reproducing kernel}
\overline{K_{p,\psi}(z)}= \overline{\langle K_{p,\psi}, K_{z,\psi} \rangle}=\langle K_{z,\psi}, K_{p,\psi} \rangle=K_{z,\psi}(p) \,\,\,\,\text{for all} \,\,  p, z\in \mathbb{C}^n
\end{equation}
and
\begin{equation*} 
||K_{p,\psi}||^2=\langle K_{p,\psi}, K_{p,\psi} \rangle= K_{p,\psi}(p) \,\,\,\,\text{for all} \,\,  p \in \mathbb{C}^n.
\end{equation*}

Let $X$ be a space whose elements are holomorphic functions defined on some domain V. For holomorphic functions  $U:V \longrightarrow \mathbb{C}$  and $\Gamma: V \longrightarrow V,$ the weighted composition operator $C_{U,\Gamma}$ is defined by
$$C_{U,\Gamma}f=U \cdot( f \circ \Gamma)$$
for $f\in X.$ If $U \equiv 1,$ then it becomes the composition operator $C_{\Gamma}f=( f \circ \Gamma)$ induced by $\Gamma.$ \\

If $C_{U,\Gamma}$ is a weighted composition operator on $F_{\psi}^2,$ then 
\begin{equation*}
\langle h,C_{U,\Gamma}^* K_{z,\psi} \rangle=\langle C_{U,\Gamma} h, K_{z,\psi} \rangle= U(z) \langle h,K_{\Gamma(z),\psi} \rangle= \langle h,\overline{U(z)} K_{\Gamma(z),\psi} \rangle \,\,\text{for all} \,\, h\in F_{\psi}^2,\, z \in \mathbb{C}^n.
\end{equation*}
Thus, $C_{U,\Gamma}^* K_{z,\psi}=\overline{U(z)} K_{\Gamma(z),\psi}$ for all $z\in \mathbb{C}^n.$\\

In this paper we study the relation between the function-theoretic behaviours of $U$ and $\Gamma$ with the operator-theoretic properties of $C_{U,\Gamma}.$ Many researchers have explored these relations on different function spaces; like  Hardy space and Bergman space on the unit disk or on Fock space.

Recently, many researchers have studied and obtained the characterizations for self-adjoint and co-isometric weighted composition operators on Hilbert spaces of analytic functions.

 Zhao and Pang \cite{Zhao and Pang} and  Zhao \cite{Normal weig comp op on the fock sp} obtained characterization for self-adjoint weighted composition operator on the Fock space of complex plane and $\mathbb{C}^n,$ respectively. Zhao \cite{A note on inv weigh comp on the fock sp} obtained that the necessary and sufficient condition for weighted composition operator to be co-isometry on the Fock space of $\mathbb{C}^n$ is that, it is unitary.

 In section 2, we obtain the characterizations for adjoint of a composition and weighted composition operator to be some composition and weighted composition operator on $F_{\psi}^2,$ respectively. Further, we discuss self-adjoint of composition and weighted composition operators. In section 3, we study co-isometric composition and weighted composition operators on $F_{\psi}^2.$

\section{Self-adjoint composition and weighted composition operators} 
Recall that a bounded linear operator $S$ on a Hilbert space $H$ is self-adjoint if $S^*=S,$ where $S^{*}$ is the adjoint of $S.$ In the following result $\Gamma_{1}$ and $\Gamma_{2}$ are taken according to the characterization obtained by Qin \cite{IJPA} for bounded compositions operators $C_{\Gamma_1}$ and $C_{\Gamma_2}$ on $F_{\psi}^2,$ respectively (in particular \cite{Bdd Composition op on Fock space Cn} for Fock space of $\mathbb {C}^{n}$). In this result we obtain characterization for adjoint of a composition operator to be composition operator on $F_{\psi}^2.$

\begin{theorem} \label{Com Op isometry iff mod(a)=1}
Let $\Gamma_{1}(z)=C_{1}z+D_{1}$ and $\Gamma_{2}(z)=C_{2}z+D_{2}$ for some $C_{1}, C_{2} \in \mathcal{M}(n)$ with $||C_{1}||\leq 1, ||C_{2}||\leq 1$ and $D_{1}, D_{2} \in \mathbb{C}^n$ such that $C_{\Gamma_{1}}$ and $C_{\Gamma_{2}}$ are bounded on $F_{\psi}^2,$ respectively. Then, the adjoint of a composition operator $C_{\Gamma_{1}}$ is composition operator $C_{\Gamma_{2}}$ if and only if 
$$ D_{1}=0, D_{2}=0\,\, \text{and}\,\, C_{1}^*=C_{2}.$$
\end{theorem}
\begin{proof} 
Let $C_{\Gamma_{1}}^{*}=C_{\Gamma_{2}}.$  Then, for all $z,w \in\mathbb{C}^n $
\begin{align}
(C_{\Gamma_{1}}^{*}K_{z,\psi})(w)&=(C_{\Gamma_{2}}K_{z,\psi})(w) \notag \\
K_{\Gamma_{1}(z),\psi}(w)&=K_{z,\psi}(\Gamma_{2}(w)). \label{selfadjoint comp eq 1}
\end{align}
Since $\Gamma_{1}(0)=D_{1}$ putting $z=0$ in \eqref{selfadjoint comp eq 1}, for all $w\in \mathbb{C}^n$ we get
\begin{align*}
K_{\Gamma_{1}(0),\psi}(w)=\frac{1}{c_{n-1}},
\end{align*}
or
\begin{align*}
\frac{1}{(n-1)!}\sum_{r=n-1}^{\infty} \frac{r(r-1) \cdots(r-n+2)}{c_{r}} \langle w, D_{1} \rangle^{r-n+1} =\frac{1}{c_{n-1}}.
\end{align*}  
Since $c_{r}>0$ for all $r \geq 0$ and the above equation holds for all $w\in \mathbb{C}^n$ in particular, if we take $w=D_{1}$ then we get $D_{1}=0.$ Putting $w=0$ in equation \eqref{selfadjoint comp eq 1}, for all $z\in \mathbb{C}^n$  we get
\begin{align*}
\frac{1}{c_{n-1}} &= K_{z,\psi}(\Gamma_{2}(0)) 
\end{align*}
or
\begin{align*}
\frac{1}{c_{n-1}}&=\frac{1}{(n-1)!}\sum_{r=n-1}^{\infty} \frac{r(r-1) \cdots(r-n+2)}{c_{r}} \langle D_{2}, z \rangle^{r-n+1}.
\end{align*}
Thus, $D_{2}=0.$ Since $D_{1}=0$ and $D_{2}=0,$ therefore, $\Gamma_{1}(z)=C_{1}z$ and $\Gamma_{2}(z)=C_{2}z,$ respectively. From equation \eqref{selfadjoint comp eq 1} it follows that, for all $z,w \in\mathbb{C}^n$  
\begin{align*}  
 \frac{1}{(n-1)!}\sum_{r=n-1}^{\infty} \frac{r(r-1) \cdots(r-n+2)}{c_{r}} \langle w, C_{1}z \rangle^{r-n+1}&= \frac{1}{(n-1)!}\sum_{r=n-1}^{\infty} \frac{r(r-1) \cdots(r-n+2)}{c_{r}} \langle C_{2}w, z \rangle^{r-n+1}, 
\end{align*}
or
\begin{align*} 
 \sum_{r=n-1}^{\infty} \frac{r(r-1) \cdots(r-n+2)}{c_{r}} \langle C_{1}^*w, z \rangle^{r-n+1}&= \sum_{r=n-1}^{\infty} \frac{r(r-1) \cdots(r-n+2)}{c_{r}} \langle C_{2}w, z \rangle^{r-n+1}, 
 \end{align*}
Thus, $C_{1}^*w=C_{2}w$ for all $w\in \mathbb{C}^n$ and hence $C_{1}^*=C_{2} $\\ 
 Conversely, let $\Gamma_{1}(z)=C_{1}z, \Gamma_{2}(z)=C_{2}z$ and $C_{1}^*=C_{2} .$ From the boundedness of $C_{\Gamma_{1}}$ and $C_{\Gamma_{2}}$ on $F_{\psi}^2$ and set $Span\{K_{z,\psi} : z \in\mathbb{C}^n \}$ being dense in $F_{\psi}^2,$ to prove $C_{\Gamma_{1}}^*=C_{\Gamma_{2}}$ it is enough to prove that $C_{\Gamma_{1}}^*K_{z,\psi}=C_{\Gamma_{2}}K_{z,\psi}$  for all $ z\in\mathbb{C}^n.$ For all $z,w \in\mathbb{C}^n$ 
\begin{align*}
C_{\Gamma_{1}}^*K_{z,\psi}(w)& = K_{\Gamma_{1}(z),\psi}(w) \\
                             &= \frac{1}{(n-1)!}\sum_{r=n-1}^{\infty} \frac{r(r-1) \cdots(r-n+2)}{c_{r}} \langle w, C_{1}z \rangle^{r-n+1} \\
                             &= \frac{1}{(n-1)!}\sum_{r=n-1}^{\infty} \frac{r(r-1) \cdots(r-n+2)}{c_{r}} \langle C_{1}^*w, z \rangle^{r-n+1} \\
                             &= \frac{1}{(n-1)!}\sum_{r=n-1}^{\infty} \frac{r(r-1) \cdots(r-n+2)}{c_{r}} \langle  C_{2}w,z \rangle^{r-n+1} \\
                             &= K_{z,\psi}(\Gamma_{2}(w)) \\
                             &=(C_{\Gamma_{2}}K_{z,\psi})(w).
\end{align*} 
Thus, $C_{\Gamma_{1}}^*=C_{\Gamma_{2}}.$ 
\end{proof}
The following result characterizes self-adjoint composition operator on $F_{\psi}^2:$
\begin{corollary}
Let $\Gamma(z)=Cz+D$ for some $C \in \mathcal{M}(n)$ with $||C||\leq 1$ and $D \in \mathbb{C}^n$  such that $C_{\Gamma}$ is bounded on $F_{\psi}^2.$ Then, $C_{\Gamma}$ is self-adjoint if and only if 
$$ D=0\,\, \text{and}\,\, C^*=C.$$
\end{corollary}
We have the following example of self-adjoint composition operator on $F_{\psi}^2:$ 
\begin{example}
Let $C=(c_{ij})_{i,j=1}^{n}$ be $n \times n$ matrix where $c_{ij}=0$ except $c_{11}=1$ and let $\Gamma(z)=Cz=(z_{1},0,\cdots, 0)$ for $z=(z_{1},z_{2},\cdots,z_{n}) \in \mathbb{C}^n.$ Then, 
$$ ||C||^2=\sup_{(z_{1},z_{2},\cdots,z_{n}) \neq 0} \frac{|Cz|^2}{|z|^2}=\sup_{(z_{1},z_{2},\cdots,z_{n}) \neq 0} \frac{|z_{1}|^2}{|z_{1}|^2+ \cdots +|z_{n}|^2} \leq 1 $$
and
 
$$ ||C||^2 =\sup_{(z_{1},z_{2},\cdots,z_{n}) \neq 0} \frac{|Cz|^2}{|z|^2} \geq \frac{|C(1,0,\cdots,0)|^2}{|(1,0,\cdots,0)|^2}=1. $$
        
Thus, $||C||=1.$ For $w=(w_{1},w_{2},\cdots,w_{n})$ and $z=(z_{1},z_{2},\cdots,z_{n}) \in \mathbb{C}^n$
\begin{align*}
(C_{\Gamma}^{*}K_{z,\psi})(w)&=\frac{1}{(n-1)!}\sum_{r=n-1}^{\infty} \frac{r(r-1) \cdots(r-n+2)}{c_{r}} \langle w, \Gamma(z) \rangle^{r-n+1} \\
                &=\frac{1}{(n-1)!}\sum_{r=n-1}^{\infty} \frac{r(r-1) \cdots(r-n+2)}{c_{r}} ( w_{1} \bar{z_{1}} )^{r-n+1} \\
                &=\frac{1}{(n-1)!}\sum_{r=n-1}^{\infty} \frac{r(r-1) \cdots(r-n+2)}{c_{r}} \langle \Gamma(w), z \rangle^{r-n+1} \\
                &=(C_{\Gamma}K_{z,\psi})(w).                 
\end{align*}
Thus, $C_{\Gamma}$ is self-adjoint on $F_{\psi}^2.$ 
\end{example}

  
\begin{theorem} \label{thm necessary condition on adjoint weig com op}
Let $\Gamma_{1}(z)=C_{1}z+D_{1}$ and $\Gamma_{2}(z)=C_{2}z+D_{2}$ for some $C_{1}, C_{2} \in \mathcal{M}(n), D_{1}, D_{2} \in \mathbb{C}^n$ with $||C_{1}||\leq 1, ||C_{2}||\leq 1$ and $U_{1}, U_{2}$ be holomorphic functions on $\mathbb{C}^n$ such that $C_{U_{1},\Gamma_{1}}$ and $C_{U_{2},\Gamma_{2}}$ are bounded on $F_{\psi}^2,$ respectively. Then, the adjoint of $C_{U_{1},\Gamma_{1}}$ is $C_{U_{2},\Gamma_{2}}$ if and only if 
\begin{equation} \label{first condition for adjoint of weighted is weighted}
U_{1}(z)=c_{n-1}U_{1}(0)K_{\Gamma_{2}(0),\psi}(z) \,\, \text{and} \,\, U_{2}(z)=c_{n-1}\overline{U_{1}(0)}K_{\Gamma_{1}(0),\psi}(z) \,\, \text{for all} \,\, z \in \mathbb{C}^n
\end{equation}
and exactly one of the following holds
\begin{itemize}
\item[(i)] If $U_{1}(0)=0$ then  $U_{1}\equiv 0,$ $U_{2}\equiv 0$ and consequently $C_{U_{1},\Gamma_{1}} \equiv 0,$ $C_{U_{2},\Gamma_{2}} \equiv 0.$ \\
\item[(ii)] If $U_{1}(0) \neq 0$ then 
\begin{equation} \label{adjoint of weghted com op}
K_{z,\psi}(\Gamma_{2}(0)) K_{\Gamma_{1}(z),\psi}(w)= K_{\Gamma_{1}(0),\psi}(w) K_{z,\psi}(\Gamma_{2}(w))  \,\,\,\, \text{for all}\,\, z,w\in \mathbb{C}^n. 
\end{equation}
\end{itemize}
\end{theorem} 
\begin{proof}
Let $C_{U_{1},\Gamma_{1}}^*=C_{U_{2},\Gamma_{2}}.$ Then, for all $z,w\in \mathbb{C}^n$
\begin{align*}
(C_{U_{1},\Gamma_{1}}^*K_{z,\psi})(w)&=(C_{U_{2},\Gamma_{2}}K_{z,\psi})(w) \notag
\end{align*}
or
\begin{align}
\overline{U_{1}(z)} K_{\Gamma_{1}(z),\psi}(w)&=U_{2}(w)K_{z,\psi}(\Gamma_{2}(w)). \label{adjoint of weigh comp is wei adjoint 1}
\end{align}
Putting $w=0$ we get  
\begin{align*}
\overline{U_{1}(z)} &=c_{n-1}  U_{2}(0)K_{z,\psi}(\Gamma_{2}(0)) \,\,\,\, \text{for all}\, z\in \mathbb{C}^n. 
\end{align*}
Putting $z=0$ we obtain $\overline{U_{1}(0)} =U_{2}(0).$ Thus, 
\begin{equation}\label{adjoint of weigh comp is wei adjoint 2}
\overline{U_{1}(z)} =c_{n-1} \overline{U_{1}(0)}\, K_{z,\psi}(\Gamma_{2}(0)) \,\,\,\,  \text{for all}\,\, z\in \mathbb{C}^n.
\end{equation}
Taking conjugate on both the sides, we obtain $U_{1}(z) =c_{n-1} U_{1}(0) \, K_{\Gamma_{2}(0),\psi}(z).$
Equations (\ref{adjoint of weigh comp is wei adjoint 1}) and (\ref{adjoint of weigh comp is wei adjoint 2}) together imply that
\begin{equation} \label{adjoint of weigh comp is wei adjoint 3}
c_{n-1}  \overline{U_{1}(0)}\, K_{z,\psi}(\Gamma_{2}(0)) K_{\Gamma_{1}(z),\psi}(w)=U_{2}(w)K_{z,\psi}(\Gamma_{2}(w)) \,\,\,\,  \text{for all}\,\, z,w\in \mathbb{C}^n.
\end{equation}
Putting $z=0$ we get 
\begin{align} \label{adjoint of weigh comp is wei adjoint 4}
c_{n-1} \overline{U_{1}(0)}    K_{\Gamma_{1}(0),\psi}(w)&=U_{2}(w) \,\,\,\,  \text{for all}\,\, w\in \mathbb{C}^n.
\end{align}
Equations (\ref{adjoint of weigh comp is wei adjoint 3}) and (\ref{adjoint of weigh comp is wei adjoint 4}) together imply that
\begin{equation} \label{adjoint of weigh comp is wei adjoint 5}
\overline{U_{1}(0)} K_{z,\psi}(\Gamma_{2}(0)) K_{\Gamma_{1}(z),\psi}(w)=\overline{U_{1}(0)} K_{\Gamma_{1}(0),\psi}(w) K_{z,\psi}(\Gamma_{2}(w))  \,\,\,\, \text{for all}\,\, z,w\in \mathbb{C}^n.
\end{equation}
Case I. If $U_{1}(0)=0$ then from equations (\ref{adjoint of weigh comp is wei adjoint 2}) and (\ref{adjoint of weigh comp is wei adjoint 4}) we get that $U_{1}\equiv 0$ and $U_{2}\equiv 0.$\\
Case II. If $U_{1}(0)\neq 0$ then equation (\ref{adjoint of weigh comp is wei adjoint 5}) implies 
\begin{equation*} \label{adjoint of weigh comp is wei adjoint 6}
 K_{z,\psi}(\Gamma_{2}(0)) K_{\Gamma_{1}(z),\psi}(w)= K_{\Gamma_{1}(0),\psi}(w) K_{z,\psi}(\Gamma_{2}(w))  \,\,\,\, \text{for all}\,\, z,w\in \mathbb{C}^n.
\end{equation*}
Conversely, let $U_{1}$ and $U_{2}$ be as it is in equation \eqref{first condition for adjoint of weighted is weighted}. If $U_{1}(0)=0,$ then $U_{1}\equiv 0,$ $U_{2}\equiv 0$ and the result holds trivially. Now, assume that $U_{1}(0)\neq 0$ and equation \eqref{adjoint of weghted com op} is satisfied. For all $z,w \in \mathbb{C}^n$
\begin{align*}
(C_{U_{1},\Gamma_{1}}^*K_{z,\psi})(w)&=\overline{U_{1}(z)} K_{\Gamma_{1}(z),\psi}(w)  \\
             &=c_{n-1} \overline{U_{1}(0)}\, K_{z,\psi}(\Gamma_{2}(0)) K_{\Gamma_{1}(z),\psi}(w). 
 \end{align*}
Using equation \eqref{adjoint of weghted com op} for all $z,w \in \mathbb{C}^n$ we get       
 \begin{align*}            
(C_{U_{1},\Gamma_{1}}^*K_{z,\psi})(w) &=c_{n-1} \overline{U_{1}(0)} \, K_{\Gamma_{1}(0),\psi}(w) K_{z,\psi}(\Gamma_{2}(w)) \\
             &=U_{2}(w) K_{z,\psi}(\Gamma_{2}(w)) \\
             &=(C_{U_{2},\Gamma_{2}}K_{z,\psi})(w) .
\end{align*}
Hence, $C_{U_{1},\Gamma_{1}}^*K_{z,\psi}=C_{U_{2},\Gamma_{2}}K_{z,\psi}$ for all $z\in \mathbb{C}^n$, therefore, $C_{U_{1},\Gamma_{1}}^*=C_{U_{2},\Gamma_{2}}.$
\end{proof}
 
\begin{theorem} \label{thm general case self adjoint  C z+D Both C,D nonzero}
Let $\Gamma(z)=Cz+D$ for some invertible matrix $ C \in \mathcal{M}(n), D\neq 0 \in \mathbb{C}^n$ with $||C||\leq 1$ and $U$ be holomorphic function on $\mathbb{C}^n$ such that $C_{U,\Gamma}$ is bounded on $F_{\psi}^2.$ If weighted composition operator $C_{U,\Gamma}$ is self-adjoint, then
$$ U(z)=\alpha  K_{\Gamma(0),\psi}(z)\,\, \text{where} \,\, \alpha=c_{n-1} \overline{U(0)} \,\, \text{is a real number} $$
and exactly one of the following holds
\begin{itemize}
\item[(i)] If $U(0)=0,$ then $U\equiv 0$ and consequently $C_{U,\Gamma} \equiv 0.$ \\
\item[(ii)] If $U(0)\neq 0,$ then $C^{*}=C$ only if it satisfy equation \eqref{general case self adjoint thm last eq}.
\end{itemize}
\end{theorem}
\textbf{Proof.} Let $C_{U,\Gamma}$ be self-adjoint on $F_{\psi}^2.$  Then, taking $\Gamma_{1}=\Gamma=\Gamma_{2}$ and $U_{1}=U=U_{2}$ in Theorem \ref{thm necessary condition on adjoint weig com op}, equation \eqref{first condition for adjoint of weighted is weighted}  implies  
\begin{align*}
 U(z)&=c_{n-1} \overline{U(0)} K_{\Gamma(0),\psi}(z).  
\end{align*}
Putting $z=0,$ we get $U(0)= \overline{U(0)}$. Since $c_{n-1}$ is positive real number, therefore,
\begin{equation*} 
 U(z)=\alpha  K_{\Gamma(0),\psi}(z)\,\,\text{for all}\,\, z\in \mathbb{C}^n\,\, \text{where} \,\, \alpha=c_{n-1} \overline{U(0)} \,\, \text{and} \,\, \bar{\alpha}=\alpha.
\end{equation*}
Case I. For $U(0)=0$, clearly $U\equiv 0.$\\
Case II. For $U(0)\neq 0,$ equation \eqref{adjoint of weghted com op} implies 
\begin{align}
  K_{z,\psi}(\Gamma(0)) K_{\Gamma(z),\psi}(w)&=  K_{\Gamma(0),\psi}(w) K_{z,\psi}(\Gamma(w)) \,\,\,\, \text{for all}\,\, z,w\in \mathbb{C}^n. \notag
\end{align} 
Putting $z=-C^{-1}D,$ $w=-C^{-1}D$ and using the fact that $\Gamma(0)=D,$ $\Gamma\big(-C^{-1}D \big)=0$ and $(C^{-1})^*=(C^{*})^{-1}$  we get 
\begin{align}
\sum_{r=n-1}^{\infty} \frac{r(r-1) \cdots(r-n+2)}{c_{r}} \langle D, -C^{-1}D \rangle^{r-n+1} &=\sum_{r=n-1}^{\infty} \frac{r(r-1) \cdots(r-n+2)}{c_{r}} \langle -C^{-1}D, D \rangle^{r-n+1} \notag\\
 \sum_{r=n-1}^{\infty} \frac{r(r-1) \cdots(r-n+2)}{c_{r}} \langle (C^{*})^{-1}D,D \rangle^{r-n+1} &=\sum_{r=n-1}^{\infty} \frac{r(r-1) \cdots(r-n+2)}{c_{r}} \langle C^{-1}D, D \rangle^{r-n+1}. \label{general case self adjoint thm last eq}
\end{align} 
Thus, $C^{*}=C$ only if it satisfy equation \eqref{general case self adjoint thm last eq}.\\ 

In Theorem \ref{thm general case self adjoint  C z+D Both C,D nonzero} for a matrix $C$ to satisfy equation \eqref{general case self adjoint thm last eq} it is not necessary for $C$ to be Hermitian matrix, that is, $C^*=C.$ This can be verified from the following example for $n=2$.

\begin{example}
Let $C=
\begin{pmatrix}
0 & \frac{1}{2} \\
-\frac{1}{2} & 0 
\end{pmatrix}$ and  
 $D=\begin{pmatrix}
1 \\
0\\
\end{pmatrix}
$ in Theorem \ref{thm general case self adjoint  C z+D Both C,D nonzero}. First we will show that $||C||\leq \frac{1}{\sqrt{2}}.$ Clearly, $Cz=(\frac{1}{2}z_2, -\frac{1}{2}z_{1})$ for $z=(z_{1},z_{2}) \in \mathbb{C}^2.$ Now consider
$$|Cz|^2=\frac{1}{4}|z_{1}|^2+\frac{1}{4}|z_{2}|^2 < \frac{1}{2} (|z_{1}|^2+|z_{2}|^2)=\frac{1}{2}|z|^2 \,\,\, \text{for} \,\, z=(z_{1},z_{2}) \in \mathbb{C}^2. $$
Thus,$$ ||C||^2=\sup_{(z_{1},z_{2}) \neq 0} \frac{|Cz|^2}{|z|^2} \leq \frac{1}{2}. $$
Also, $C^{-1}=\begin{pmatrix}
0 & -2 \\
2 & 0 
\end{pmatrix}$ 
and  
$(C^{*})^{-1}=\begin{pmatrix}
0 & 2 \\
-2 & 0 
\end{pmatrix}.$ Since $C^{-1}D=(0,2)$ and $(C^{*})^{-1}D=(0,-2),$ therefore,
$$\langle (C^{*})^{-1}D,D \rangle=0= \langle C^{-1}D, D \rangle. $$
Hence, $C$ is a non-Hermitian matrix satisfying equation \eqref{general case self adjoint thm last eq}.
\end{example}

The following two results can be obtained in a similar manner as of Theorem \ref{thm general case self adjoint  C z+D Both C,D nonzero}: 
\begin{proposition} \label{thm general case self adjoint  C z}
Let $\Gamma(z)=C z$ for $C \in \mathcal{M}(n)$ with $||C||\leq 1$ and $U$ be holomorphic function on $\mathbb{C}^n$ such that $C_{U,\Gamma}$ is bounded on $F_{\psi}^2.$ Then, the  weighted composition operator $C_{U,\Gamma}$ is self-adjoint if and only if
$$ U(z)=\overline{U(0)} $$
and exactly one of the following holds
\begin{itemize}
\item[(i)] If $U(0)=0,$ then $U\equiv 0$ and consequently $C_{U,\Gamma} \equiv 0.$ 
\item[(ii)] If $U(0)\neq 0,$ then  $C$ is self adjoint.
\end{itemize}
\end{proposition}
\begin{proposition} \label{self adjoint constant case D}
Let $\Gamma(z)=D$ for  $D\in \mathbb{C}^n$  and $U$ be holomorphic function on $\mathbb{C}^n$ such that $C_{U,\Gamma}$ is bounded on $F_{\psi}^2.$ Then, the weighted composition operator $C_{U,\Gamma}$ is self-adjoint if and only if
 $$ U(z)=\alpha  K_{\Gamma(0)}(z) \,\,\, \text{for all}\,\, z \in\, \mathbb{C}^n,$$ 
 where $\alpha=c_{n-1} \overline{U(0)}$ is a real number.
\end{proposition}

\section{Co-isometry composition and weighted composition operators}
Recall that a bounded linear operator $S$ on a Hilbert space $H$ is co-isometry if $S^*$ is isometry or equivalently, $SS^*=I_{H}$ or  $||S^{*}x||=||x||$  or  $\langle S^{*}x, S^{*}y \rangle=\langle x, y \rangle$ for all $x, y \in H.$ The following result gives the necessary and sufficient condition for composition operator to be co-isometry:
\begin{theorem}  \label{composition op co-isometry}
Let $\Gamma(z)=C z-D$ for some matrix $ C \in \mathcal{M}(n), D \in \mathbb{C}^n$ with $||C||\leq 1$ such that $C_{\Gamma}$ is bounded on $F_{\psi}^2.$ Then, $C_{\Gamma}$ is co-isometry if and only if $D=0$ and $ C$ is unitary. 
\end{theorem}
\textbf{Proof.} Let $C_{\Gamma}$ be co-isometry on $F_{\psi}^2.$ Then  $C_{\Gamma} C_{\Gamma}^*=I_{F_{\psi}^2}$ implies $C_{\Gamma}$ is onto and so $\Gamma$ cannot be constant, that is, matrix $C$ is nonzero matrix. First, we will prove that matrix $C$ to be invertible.
To prove $C \in \mathcal{M}(n)$ is invertible, it is sufficient to prove that $\Gamma$ is injective. Let $z,w \in \mathbb{C}^n$ such that $\Gamma(z)=\Gamma(w).$ Now consider 
$ C_{\Gamma}^*(K_{z,\psi}-K_{w,\psi})=K_{\Gamma(z),\psi}-K_{\Gamma(w),\psi}=0$
Thus, $(K_{z,\psi}-K_{w,\psi}) \in Ker(C_{\Gamma}^*).$ Since $C_{\Gamma}^*$ is isometry so injective, therefore, $K_{z,\psi}=K_{w,\psi}$ which implies $z=w.$
For all $z,w\in \mathbb{C}^n$ we have
\begin{align*}
(C_{\Gamma} C_{\Gamma}^*K_{z,\psi})(w)&=K_{z,\psi}(w), \notag 
\end{align*}
or
\begin{align}
\sum_{r=n-1}^{\infty} \frac{r(r-1) \cdots(r-n+2)}{c_{r}} \langle \Gamma(w), \Gamma(z) \rangle^{r-n+1} &=\sum_{r=n-1}^{\infty} \frac{r(r-1) \cdots(r-n+2)}{c_{r}} \langle w, z \rangle^{r-n+1}. \label{Com Coisometry eq 1}  
 \end{align} 
 Since $\Gamma(C^{-1}D)=0$ putting $w=C^{-1}D$ for all $z\in \mathbb{C}^n$  we get 
\begin{align*} 
\frac{1}{(n-1)!}\sum_{r=n-1}^{\infty} \frac{r(r-1) \cdots(r-n+2)}{c_{r}} \langle C^{-1}D, z \rangle^{r-n+1}=\frac{1}{c_{n-1}}, \notag 
\end{align*} 
equivalently
\begin{align*}
\frac{1}{(n-1)!}\sum_{r=n}^{\infty} \frac{r(r-1) \cdots(r-n+2)}{c_{r}} \langle C^{-1}D, z \rangle^{r-n+1}=0.
\end{align*}
Thus, $C^{-1}D=0$ and hence $D=0.$ This implies, $\Gamma(z)=Cz.$ Putting $z=w$ in \eqref{Com Coisometry eq 1} for all $z\in \mathbb{C}^n$ we get 
\begin{align*}
 \frac{1}{(n-1)!}\sum_{r=n-1}^{\infty} \frac{r(r-1) \cdots(r-n+2)}{c_{r}} \langle Cz, Cz \rangle^{r-n+1} &=\frac{1}{(n-1)!}\sum_{r=n-1}^{\infty} \frac{r(r-1) \cdots(r-n+2)}{c_{r}} \langle z, z \rangle^{r-n+1} \\
\sum_{r=n-1}^{\infty} \frac{r(r-1) \cdots(r-n+2)}{c_{r}} \langle C^*Cz, z \rangle^{r-n+1} &=\sum_{r=n-1}^{\infty} \frac{r(r-1) \cdots(r-n+2)}{c_{r}} \langle z, z \rangle^{r-n+1}  
 \end{align*}
 Thus, $C^*Cz=z$ for all $z\in \mathbb{C}^n$ and since $C\in \mathcal{M}(n), C$ is unitary.\\
 Conversely, let $\Gamma(z)=Cz$ and $C$ is unitary. For all $z,w \in \mathbb{C}^n$
\begin{align*} 
 (C_{\Gamma} C_{\Gamma}^*K_{z,\psi})(w)&=K_{\Gamma(z),\psi}(\Gamma(w)) \\
                                       &= \frac{1}{(n-1)!}\sum_{r=n-1}^{\infty} \frac{r(r-1) \cdots(r-n+2)}{c_{r}} \langle Cw, Cz \rangle^{r-n+1} \\
                                       &= \frac{1}{(n-1)!}\sum_{r=n-1}^{\infty} \frac{r(r-1) \cdots(r-n+2)}{c_{r}} \langle C^* Cw, z \rangle^{r-n+1} \\
                                       &= \frac{1}{(n-1)!}\sum_{r=n-1}^{\infty} \frac{r(r-1) \cdots(r-n+2)}{c_{r}} \langle w, z \rangle^{r-n+1} \\
                                       &=K_{z,\psi}(w).
\end{align*}
Thus, $C_{\Gamma}$ is co-isometry.\\

Note, for $\Gamma(z)=Cz$ and unitary matrix $C \in \mathcal{M}(n),$ since $|Cz|^2=|z|^2$ and $Cz=w$ has solution for each $w\in  \mathbb{C}^n ,$ therefore, we have
\begin{align*} 
||C_{\Gamma}f||^2 &=\int_{\mathbb{C}^n} |(f \circ \Gamma)(z)|^2 e^{-\psi(|z|^2) } d\nu(z) \\
                  &=\int_{\mathbb{C}^n} |f(C z)|^2 e^{-\psi(|z|^2) } d\nu(z) \\
                  &=\int_{\mathbb{C}^n} |f(C z)|^2 e^{-\psi(|Cz|^2) } d\nu(z) \\
                  &=\int_{\mathbb{C}^n} |f(z)|^2 e^{-\psi(|z|^2) } d\nu(z) \\
                  &=||f||^2
\end{align*} 
Thus, $C_{\Gamma}$ is isometry.\\

The above discussion along with Theorem \ref{composition op co-isometry} lead to the following result:
\begin{corollary}  \label{composition op unitary}
Let $\Gamma(z)=C z-D$ for some matrix $ C \in \mathcal{M}(n), D \in \mathbb{C}^n$ with $||C||\leq 1$ such that $C_{\Gamma}$ is bounded on $F_{\psi}^2.$ Then, composition operator $C_{\Gamma}$ is co-isometry if and only if it is a unitary operator.
\end{corollary}

\begin{lemma} \label{weighted comp Az theorem}
Let $\Gamma(z)=C z$ for matrix $ C \in \mathcal{M}(n)$ with $||C||\leq 1$ and $U$ be holomorphic function on $\mathbb{C}^n$ so that $C_{U,\Gamma}$ is bounded on $F_{\psi}^2.$ Then, $C_{U,\Gamma}$ is co-isometry if and only if
$$U(z)=\frac{1}{\overline{U(0)}}, \,\, |U(z)|=1 \,\, \text{and} \,\, C \,\, \text{is unitary}.$$
\end{lemma} 
\textbf{Proof.} Let  $C_{\Gamma}$ be co-isometry. Then, for all $z, w \in \mathbb{C}^n$
$$(C_{U,\Gamma} C_{U,\Gamma}^*K_{z,\psi})(w)=K_{z,\psi}(w),$$
or
\begin{align} \label{weighted comp AZ eq 1}
U(w) \overline{U(z)} K_{\Gamma(z),\psi}(\Gamma(w))=K_{z,\psi}(w). 
\end{align}
Putting $z=0$ we get
$$U(w) \overline{U(0)} K_{\Gamma(0),\psi}(\Gamma(w))=\frac{1}{c_{n-1}}. $$
It follows that $U(0) \neq 0$ otherwise if $U(0)= 0$ then equation will become zero on left side but right side is non-zero because $c_{n-1}> 0.$ Since $\Gamma(0)=0$ and $K_{\Gamma(0),\psi} =\frac{1}{c_{n-1}},$ therefore,  we get 
$$ U(w)=\frac{1}{\overline{U(0)}} \,\, \text{for all} \,\, w \in \mathbb{C}^n.$$
Putting $w=0$ we get $|U(0)|=1$ so 
\begin{align} \label{weighted comp AZ eq 2}
|U(w)|=1 \,\, \text{for all} \,\, w \in \mathbb{C}^n.
\end{align}
Putting $z=w$ in equation \eqref{weighted comp AZ eq 1} and using equation \eqref{weighted comp AZ eq 2}we get  
$$ K_{\Gamma(z),\psi}(\Gamma(z))=K_{z,\psi}(z),$$ 
or
\begin{align*}
\sum_{r=n-1}^{\infty} \frac{r(r-1) \cdots(r-n+2)}{c_{r}} \langle Cz, Cz \rangle^{r-n+1} &=\sum_{r=n-1}^{\infty} \frac{r(r-1) \cdots(r-n+2)}{c_{r}} \langle z, z \rangle^{r-n+1} 
\end{align*}
Thus, $\langle Cz, Cz \rangle=\langle z, z \rangle$ for all $ z\in \mathbb{C}^n$ and since $C\in \mathcal{M}(n), C$ is unitary. \\
Conversely, let $ U(w)=\dfrac{1}{\overline{U(0)}}, |U(w)|=1$ and $C$ be unitary. Then, for all $z, w \in   \mathbb{C}^n$ we get
\begin{align*}
(C_{U,\Gamma} C_{U,\Gamma}^*K_{z,\psi})(w) &=  U(w) \overline{U(z)} K_{\Gamma(z),\psi}(\Gamma(w))\\
                                           &= K_{\Gamma(z),\psi}(\Gamma(w)) \\
                                           &=\frac{1}{(n-1)!}\sum_{r=n-1}^{\infty} \frac{r(r-1) \cdots(r-n+2)}{c_{r}} \langle Cw, Cz \rangle^{r-n+1} \\
                                           &=\frac{1}{(n-1)!}\sum_{r=n-1}^{\infty} \frac{r(r-1) \cdots(r-n+2)}{c_{r}} \langle w, z \rangle^{r-n+1} \\
                                           &= K_{z,\psi}(w)
\end{align*}
Thus, $C_{U,\Gamma}$ is co-isometry. \\

It is well known that if $V$ is a unitary operator on a Hilbert space $H$ and $\beta$ is a scalar  with $|\beta|=1,$ then $\beta V$ is a unitary operator on $H.$ Combining Theorems \ref{composition op co-isometry}, \ref{weighted comp Az theorem} and Corollary \ref{composition op unitary} we get the following result:
\begin{corollary}
Let $\Gamma(z)=C z$ for matrix $ C \in \mathcal{M}(n), B \in \mathbb{C}^n$ with $||C||\leq 1$ and $U$ be entire function on $\mathbb{C}^n$ so that $C_{U,\Gamma}$ is bounded on $F_{\psi}^2.$ Then, $C_{U,\Gamma}$ is co-isometry if and only if it is unitary.
\end{corollary}

\begin{remark}
If $\Gamma$ is a constant function on $\mathbb{C}^n$ say $\Gamma \equiv D$ then for any holomorphic function $U$ on $\mathbb{C}^n$ bounded composition operator $C_{U,\Gamma}$ cannot be co-isometry on $F_{\psi}^2.$ It is easy to see that the necessary condition for $C_{U,\Gamma}$ to be co-isometry is that for all $z \in \mathbb{C}^n$ 
$$ U(z)=\frac{1}{c_{n-1} K_{D,\psi} (D)\overline{U(0)}}, $$
that is $U$ is a constant function on $\mathbb{C}^n.$ For constant functions $U$ and $\Gamma$ since $C_{U,\Gamma}$ cannot be onto on $F_{\psi}^2$, therefore, $C_{U,\Gamma}$ cannot be co-isometry. 
\end{remark} 

\begin{theorem}
Let $\Gamma(z)=C z-D$ for some invertible matrix  $ C \in \mathcal{M}(n), D \in \mathbb{C}^n$ with $||C||\leq 1$ and $U$ be holomorphic function on $\mathbb{C}^n$ such that $C_{U,\Gamma}$ is bounded on $F_{\psi}^2.$ If $C_{U,\Gamma}$ is co-isometry, then for all $z\in \mathbb{C}^n$
$$U(z) =\beta \frac{K_{C^{-1}D,\psi}(z)}{||K_{C^{-1}D,\psi}||}$$
where $\beta=\dfrac{1}{\overline{U(0)} ||K_{C^{-1}D,\psi}|| },$ $|\beta|=\sqrt{c_{n-1}}$ and $|C^{-1}D|=|D|.$
\end{theorem}
\textbf{Proof.} Let  $C_{\Gamma}$ be co-isometry. Then, for all $z, w \in \mathbb{C}^n$
\begin{align} \label{weighted comp AZ-D eq 1}
U(w) \overline{U(z)} K_{\Gamma(z),\psi}(\Gamma(w))=K_{z,\psi}(w). 
\end{align} 
Putting $w=0,$ for all $z \in \mathbb{C}^n$ we get 
$$U(0) \overline{U(z)} K_{\Gamma(z),\psi}(\Gamma(0))=\frac{1}{c_{n-1}}. $$
It follows that $U(0)\neq 0,$ in fact $U(z)\neq 0$ for any $z \in \mathbb{C}^n$ because $c_{n-1}>0.$ Since $\Gamma(0)=-D$ it follows that for all $z \in \mathbb{C}^n$
\begin{align} \label{weighted comp AZ-D eq 2}
  \overline{U(z)} =\frac{1}{c_{n-1} U(0) K_{\Gamma(z),\psi}(-D)}.
  \end{align}  
Puttting $z=0$ we get
\begin{align} \label{weighted comp AZ-D eq 3} 
\overline{U(0)} =\frac{1}{c_{n-1} U(0) K_{-D,\psi}(-D)} \,\, \text{or}\, \, |U(0)|^2 =\frac{1}{c_{n-1} K_{-D,\psi}(-D)}. 
\end{align} 
Putting $w=C^{-1}D$ in equation \eqref{weighted comp AZ-D eq 1} and using $\Gamma(C^{-1}D)=0$ for all $z \in \mathbb{C}^n,$  we get 
$$U(C^{-1}D) \,  \overline{U(z)}\, \frac{1}{c_{n-1}}=K_{z,\psi}(C^{-1}D)$$
or  
\begin{align} \label{weighted comp AZ-D eq 4}
 \overline{U(z)}=c_{n-1}\frac{ K_{z,\psi}(C^{-1}D)}{U(C^{-1}D)}
\end{align}
Putting $z=0$ we get
\begin{align} \label{weighted comp AZ-D eq 5}
\overline{U(0)}=\frac{ 1}{U(C^{-1}D)} 
\end{align} 
Combining equations \eqref{weighted comp AZ-D eq 3} and \eqref{weighted comp AZ-D eq 5} we get
\begin{align} \label{weighted comp AZ-D eq 6}
K_{-D,\psi}(-D) =\frac{U(C^{-1}D)}{c_{n-1} U(0)}
 \end{align}
Equating $\overline{U(z)}$ from equations \eqref{weighted comp AZ-D eq 2}, \eqref{weighted comp AZ-D eq 4}
and using equation \eqref{weighted comp AZ-D eq 6}
\begin{align*} \label{weighted comp AZ-D eq 7} 
\frac{1}{c_{n-1}} K_{-D,\psi}(-D)=K_{\Gamma(z),\psi}(-D) K_{z,\psi}(C^{-1}D) \,\,\, \text{for all} \, z \in \mathbb{C}^n
 \end{align*}
 Putting $z=C^{-1}D$ we get 
\begin{equation} \label{weighted comp AZ-D eq 8*}
  K_{-D,\psi}(-D)=K_{C^{-1}D,\psi}(C^{-1}D) 
 \end{equation}
 or
$$ \sum_{r=n-1}^{\infty} \frac{r(r-1) \cdots(r-n+2)}{c_{r}} \langle D, D \rangle^{r-n+1}=\sum_{r=n-1}^{\infty} \frac{r(r-1) \cdots(r-n+2)}{c_{r}} \langle C^{-1}D, C^{-1}D \rangle^{r-n+1}.$$
Thus,
\begin{equation} \label{weighted comp AZ-D eq 8}                
 \langle D, D \rangle=\langle C^{-1}D, C^{-1}D \rangle.
\end{equation}
Case I. If $D=0,$ then $C^{-1}D=0.$ Since $K_{0}(z)=||K_{0}||^2=\frac{1}{c_{n-1}},$ the required result follows from Lemma \ref{weighted comp Az theorem} \\
Case II. If $D \neq 0$ then combining equations \eqref{weighted comp AZ-D eq 3}, \eqref{weighted comp AZ-D eq 4}, \eqref{weighted comp AZ-D eq 5} and  for all $z \in \mathbb{C}^n$ we get 
\begin{align} 
\overline{U(z)} &=c_{n-1}  \overline{U(0)}\, K_{z,\psi}(C^{-1}D) \notag
\\
                &=\frac{K_{z,\psi}(C^{-1}D)}{U(0)K_{-D,\psi}(-D)}. \label{weighted comp AZ-D eq 9}                
\end{align}
From equation \eqref{weighted comp AZ-D eq 8*} it follows that  
\begin{align*}
K_{-D,\psi}(-D) &=||K_{C^{-1}D,\psi}||^2
\end{align*}
Substituting this in equation \eqref{weighted comp AZ-D eq 3} and  \eqref{weighted comp AZ-D eq 9}  we get
\begin{align*}
 |U(0)| =\frac{1}{\sqrt{c_{n-1}} ||K_{C^{-1}D,\psi}||}
\end{align*} 
and
\begin{align*} 
\overline{U(z)}=\frac{K_{z,\psi}(C^{-1}D)}{U(0)||K_{C^{-1}D,\psi}||^2},
\end{align*}  
respectively. Taking conjugate on both the sides we get 
\begin{align*} 
{U(z)}=\beta \frac{K_{C^{-1}D,\psi}(z)}{||K_{C^{-1}D,\psi}||} \,\,\, \text{for all}\,\, z \in \mathbb{C}^n 
\end{align*} 
where $\beta=\dfrac{1}{\overline{U(0)} ||K_{C^{-1}D,\psi}||}, |\beta|=\sqrt{c_{n-1}}$ and from equation \eqref{weighted comp AZ-D eq 8} $|C^{-1}D|=|D|.$ 
             

\textbf{Anuradha Gupta}\\
 Department of Mathematics, Delhi College of Arts and Commerce,\\
  University of Delhi, New Delhi-110023, India.\\
  \vspace{0.2cm}
 email: dishna2@yahoo.in\\
  \textbf{Geeta Yadav} \\
  Department of Mathematics, Janki Devi Memorial College,\\
  University of Delhi, New Delhi-110060, India.\\
  \vspace{0.2cm} 
  email: ageetayadav@gmail.com \\
  Corresponding Author: Geeta Yadav
\end{document}